\documentclass[reqno]{amsart}
\usepackage{amsfonts,latexsym,amssymb,amsmath,amsthm,color}
\usepackage{enumerate}
\usepackage{xcolor}

\usepackage{soul}
\usepackage[pagebackref]{hyperref}

\usepackage[pdftex]{graphicx}
\usepackage[all]{xy}
\usepackage{tikz}
\usetikzlibrary{shapes}
\usetikzlibrary{plotmarks}

\newtheorem{teorema}{Theorem}[section]
\newtheorem{propo}[teorema]{Proposition}
\newtheorem{lema}[teorema]{Lemma}
\newtheorem{coro}[teorema]{Corollary}

\newtheorem{eje}{Example}
\theoremstyle{definition}
\newtheorem{defin}[teorema]{Definition}

\theoremstyle{remark}

\newtheorem{ejemplo}[teorema]{Example}

\newcommand{\mult}{{\rm{mult}}}
\newcommand{\ord}{{\rm{ord}}}
\newcommand{\sop}{{\hbox{\rm supp}}}
\newcommand{\grad}{{\hbox{\rm deg}}}
\newcommand{\dd}{{\hbox{\rm d}}}
\newcommand{\PH}{{\hbox{\rm PH}}}
\newcommand{\fdp}{\hfill $\square$}

\begin{document}

\title{Characterization of second type plane foliations using Newton polygons}

\author{Percy Fern\'{a}ndez-S\'{a}nchez}
\address[Percy Fern\'{a}ndez]{Dpto. Ciencias - Secci\'{o}n Matem\'{a}ticas, Pontificia Universidad Cat\'{o}lica del Per\'{u}, Av. Universitaria 1801,
San Miguel, Lima 32, Peru}
\email{pefernan@pucp.edu.pe}

\author{Evelia R. Garc\'{i}a Barroso}
\address[Evelia R. Garc\'{i}a Barroso]{Dpto. Matem\'{a}ticas, Estad\'{\i}stica e Investigaci\'on Operativa\\
Secci\'on de Matem\'aticas\\
Universidad de La Laguna. Apartado de Correos 456. 38200 La Laguna, Tenerife, Spain.}
\email{ergarcia@ull.es}

\author{Nancy Saravia-Molina}
\address[Nancy Saravia-Molina]{Dpto. Ciencias - Secci\'{o}n Matem\'{a}ticas, Pontificia Universidad Cat\'{o}lica del Per\'{u}, Av. Universitaria 1801,
San Miguel, Lima 32, Peru}
\email{nsaraviam@pucp.pe}

\thanks{First-named and third-named authors were partially supported by the Pontificia Universidad Cat\'{o}lica del Per\'{u} project VRI-DGI 2017-01-0083. Second-named author was partially supported by the Spanish grant  MTM2016-80659-P}

\date{\today}

\begin{abstract}
In this article we characterize the foliations that have the same Newton polygon that their union of formal separatrices, they are the foliations called of the second type. In the case of cuspidal foliations studied by Loray \cite{Loray}, we precise this characterization using the Poincar\'e-Hopf index. This index  also characterizes the cuspidal foliations having the same desingularization that the union of its  separatrices. Finally we give necessary and sufficient conditions when these cuspidal foliations are  generalized curves, and a characterization when they have only one separatrix.

\end{abstract}

\maketitle
\section{Introduction}
Camacho, Lins-Neto and Sad \cite{Camacho-Lins-Sad} introduced and studied the singularities of foliations of the generalized curve type, these are the foliations without saddle-nodes in their reduction of singularities. These foliations receive this name because they have a behavior similar to the union of their separatrices, where a separatrix is an irreducible analytical curve invariant for the foliation. For these foliations the Poincar\'e-Hopf index coincides with the Milnor number of the union of their separatrices (\cite{Mol-Soares} and \cite{Camacho-Lins-Sad}) and the reduction of singularities of these foliations coincides with the desingularization of the union of their separatrices.

The singularities of generalized curved type verify that their G\'omez Mont - Seade - Verjovski index \cite{Gomez Mont-Seade-Verjovsky} is zero and their Camacho - Sad index \cite{Camacho-Sad2} and the Baum-Bott index are equal \cite{Brunella1}.\\

The foliations of the second type can be thought of as a generalization of the singularities of the generalized curve type, in which we will allow the existence of formal separatrices. In order to add the formal separatrix we must admit that we have saddle-nodes in their resolution of singularities that generate formal separatrices, they could not be a corner of two divisors, nor could saddle-nodes outside the corners with weak separatrix contained in the divisor. Note that with these restrictions the singularities of a second type foliation with a single separatrix will have to be a generalized curve foliation. The singularities of second type were introduced by Mattei and Salem \cite{Mattei-Salem}. They characterized this type of singularities by means of the coincidence of the multiplicity of the foliation with the multiplicity of the union of their formal separatrices. For these singularities, the reduction of singularities coincides with the desingularization of its separatrix. It should be noted that the proof made in \cite{Camacho-Lins-Sad} to prove this property for generalized curve type foliations also proves this property for second type singularities. There are other characterizations of these singularities (see \cite{Cano2-Corral-Mol} and \cite{Fernandez-Mol}).\\

Merle \cite{Merle} gives a decomposition of the polar curve of an irreducible curve $\mathcal{C}$ that determines the topology of $\mathcal{C}$. This theorem was generalized for foliations by Rouill\'e \cite{Rouille1}, where he gives a decomposition of the polar of a foliation, of the generalized curve type, which determines the topology of its separatrix. The main ingredient for his decomposition is Dulac's Theorem \cite{Dulac}, this theorem tells us that the Newton polygon of the foliation coincides with the Newton polygon of the separatrix. Merle's theorem has been generalized for reduced curves by \cite{GarciaBarroso} and \cite{GarciaBarroso-Gwozdziewicz1}, the first of which was generalized for foliations by \cite{Corral1} and the second by \cite{Saravia}.
There are examples of foliations where their Newton polygon coincides with that of their separatrices and is not a foliation of the generalized curve type, however these foliations are of the second type (see Example \ref{segundo tipo}). In this paper we give a new characterization of the singularities of the second type in terms of the Newton polygon of their union of separatrices.

\begin{teorema}\label{resultado1}
A non-dicritical foliation is of the second type if and only if its Newton polygon coincides with that of their union of separatrices.
\end{teorema}

We will prove Theorem \ref{resultado1} in Section \ref{characterization}.

\medskip

According to Loray \cite{Loray}, a foliation with a cuspidal singularity is given by
\begin{equation}
\label{cusp}
\mathcal{F}_{\omega_{p,q,\Delta}}: \omega_{p,q,\Delta}=\dd(y^{p}-x^{q})+\Delta(x,y)(px\dd y-qy\dd x),
\end{equation}
where $p,q$ are positive natural numbers and $\Delta(x,y) \in \mathbb{C}\{x,y\}$.

\medskip

We will use Theorem \ref{resultado1} to characterize, in terms of the Poincar\'e-Hopf index for foliations, when the foliations with cuspidal singularities studied by Loray \cite{Loray} are of the second type.

\medskip

Denote by $\PH({\mathcal F})$ the Poincar\'e-Hopf index of the foliation ${\mathcal F}$. For the foliation $\dd(y^{p}-x^{q})$ we have  $\PH_{(p,q)}:=\PH(\dd(y^{p}-x^{q}))=(p-1)(q-1)$.

\begin{teorema}\label{resultado2}
Let $\mathcal{F}_{\omega_{p,q,\Delta}}$ be a cuspidal foliation as in (\ref{cusp}) and suppose that $\mathcal{F}_{\omega_{p,q,\Delta}}$ is non-dicritical. The next statements are equivalents:
\begin{itemize}
\item [$(a)$] The cuspidal foliation $\mathcal{F}_{\omega_{p,q,\Delta}}$  is of the second type.
\item [$(b)$] The intersection number $(\Delta,y^{p}-x^{q})_{0} \geq \PH_{(p,q)}-1$.
\item [$(c)$] The cuspidal foliation $\mathcal{F}_{\omega_{p,q,\Delta}}$ has the same reduction of singularities that $\dd (y^p-x^q)$.
\end{itemize}
\end{teorema}

\medskip

In general, if a foliation has the same reduction of singularities as its union of separatrices then  the  foliation is not of the second type (see Example \ref{contra1}). However, after Theorem  \ref{resultado2}, for the cuspidal foliations this property characterizes foliations of the second type.

\medskip

We also give, in the next theorem,  necessary and sufficient conditions when the cuspidal foliations are of the generalized curve type.
\begin{teorema}\label{resultado3} Let $\mathcal{F}_{\omega_{p,q,\Delta}}$ be a cuspidal foliation as in (\ref{cusp}) and suppose that $\mathcal{F}_{\omega_{p,q,\Delta}}$ is non-dicritical. We have:
\begin{enumerate}
\item [(a)] If the intersection number $(\Delta,y^{p}-x^{q})_{0}>\PH_{(p,q)}-1$, then $\mathcal{F}_{p,q,\Delta}(x,y)$ is of the generalized curve type.
\item [(b)] If  $p, q $ are coprime then the foliation $\mathcal{F}_{\omega_{p,q,\Delta}}$ is of generalized curve type if and only if $(\Delta,y^{p}-x^{q})_{0}>\PH_{(p,q)}-1$.
\end{enumerate}
\end{teorema}

We will prove Theorem \ref{resultado2} and Theorem \ref{resultado3} in Section \ref{cuspidal section}.

\section{Basic Definitions and Notations}
In order to fix the notation, we will remember basic concepts of local foliation theory and plane curves.
Denote by $\mathbb{C}[[x,y]]$ the ring of formal powers series in two variables with coefficients in $\mathbb{C}$ and $\mathbb{C}\{x,y\}$ the sub-ring of $\mathbb{C}[[x,y]]$ formed by formal powers series that converge in a neighborhood of $0 \in \mathbb{C}^{2}$. Consider the maximal ideals $\mathfrak{m}$ and $\widehat{\mathfrak{m}}$ of $\mathbb{C}\{x,y\}$ and $\mathbb{C}[[x,y]]$ respectively. The {\em order} of a power series $\widehat{h}(x,y)=\displaystyle\sum_{ij}a_{ij}x^iy^j\in
\mathbb{C}[[x,y]]$ is $\ord (\widehat{h}):=\min\{i+j\;:\;a_{ij}\neq 0\}$.\\

A {\em singular formal foliation} $\widehat{\mathcal{F}_{\omega}}$ of codimension one over $\mathbb{C}^2$  is locally given by a $1$-form $\widehat{\omega}=\widehat{A}(x,y)\dd x+ \widehat{B}(x,y)\dd y$, where $\widehat{A},\widehat{B}\in \widehat{\mathfrak{m}}$ are coprime. The power series $\widehat{A}$ and $\widehat{B}$ are called the {\em coefficients} of $\widehat{\omega}$. The {\em multiplicity} of the foliation $\widehat{\mathcal{F}_{\omega}}$ is defined as $\mult(\widehat{\omega}):=\min\{\ord(\widehat{A}),\ord(\widehat{B})\}$.\\

Let $T \subseteq \mathbb{N}^{2}$. Denote by $D(T)$ the convex hull of $(T + \mathbb{R}^{2}_{\geq0})$, where $+$ is the  Minkowski sum, and by $\mathcal{N}(T)$ the polygonal boundary of $D(T)$, which will call {\em Newton polygon} determined by $T$.\\

Let $\widehat{h}(x,y)=\displaystyle\sum_{i,j}a_{ij}x^iy^j\in \mathbb{C}[[x,y]]$. The {\em support} of $\widehat{h}$ is
\[
\sop(\widehat{h}):=\{(i,j)\in \mathbb{N}^{2}\;:\;a_{ij}\neq 0\},
\]

\noindent and the Newton polygon of $\widehat{h}$ is by definition the Newton polygon $\mathcal{N}(\sop(\widehat{h}))$.\\

Let $\widehat{\omega}=\widehat{A}(x,y)\dd x+\widehat{B}(x,y)\dd y$ be a one-form, where $\widehat{A},\widehat{B} \in \widehat{\mathfrak{m}}$. The {\em support} of $\widehat{\omega}$ is
\[
\sop(\widehat{\omega})=\sop(x\widehat{A})\cup \sop(y\widehat{B}).
\]

If we write $\widehat{\omega}=\displaystyle \sum_{i,j}\widehat{\omega}_{ij},$ where $\widehat{\omega}_{ij}=\widehat{A}_{ij}x^{i-1}y^{j}\dd x+ \widehat{B}_{ij}x^{i}y^{j-1}\dd y$, then
\[
\sop(\widehat{\omega})=\{(i,j):(\widehat{A}_{ij},\widehat{B}_{ij})\neq (0,0)\}.
\]

Let $\widehat{\mathcal{F}_{\omega}}:\widehat{\omega}=0$ be a foliation given by the one-form $\widehat{\omega}$. The Newton polygon of $\widehat{\mathcal{F}_{\omega}}$, denoted by $\mathcal{N}(\widehat{\mathcal{F}_{\omega}})$ or $\mathcal{N}(\widehat{\omega})$ is the Newton polygon $\mathcal{N}(\sop(\widehat{\omega}))$.\\

Let $\widehat{f}(x,y)\in  \mathbb{C}[[x,y]]$. We say that the $\widehat{\mathcal{S}}_{f}: \widehat{f}(x,y)=0$  is {\em invariant} by $\widehat{\mathcal{F}_{\omega}}$ if $\widehat{\omega} \wedge \dd \widehat{f}=\widehat{f}.\widehat{\eta},$ where $\widehat{\eta}$ is a two-form (i.e. $\widehat{\eta}=\widehat{g} \dd x \wedge \dd y$, for some $\widehat{g} \in \mathbb{C}[[x,y]]$). If $\widehat{\mathcal{S}}_{f}$ is irreducible then we will say that $\widehat{\mathcal{S}}_{f}$ is a {\em formal separatrix} of $\widehat{\mathcal{F}_{\omega}}:\widehat{\omega}=0$.\\

We will consider {\em non-dicritical} foliations, that is, foliations having a finite set of separatrices (see \cite[page 158 and page 165]{Camacho-Lins-Sad}). Let $(\widehat{\mathcal{S}}_{f_{j}})_{j=1}^{r}$ be the set of all formal separatrices of the non-dicritical foliation  $\widehat{\mathcal{F}_{\omega}}:\widehat{\omega}=0$. Each separatrix $\widehat{\mathcal{S}}_{f_{j}}$ corres\-ponds to an irreducible formal power series $\widehat{f_{j}}(x,y)$. Denote by $\widehat{\mathcal{S}}(\widehat{\mathcal{F}_{\omega}})$ the union $\bigcup\widehat{\mathcal{S}}_{f_{j}}$ of all separatrices of the foliation $\widehat{\mathcal{F}_{\omega}}$, which we will call {\em union of formal separatrices} of $\widehat{\mathcal{F}_{\omega}}$. In the following we will denote by $\mathcal{F}_{\omega}$ a holomorphic foliation and by $\mathcal{S}(\mathcal{F}_{\omega})$ its union of convergent separatrices.\\

The {\em dual vector field} associated to $\widehat{\mathcal{F}_{\omega}}$ is $X=\widehat{B}(x,y)\frac{\partial}{\partial x}-\widehat{A}(x,y)\frac{\partial}{\partial y}$. We say that the origin $(x,y)=(0,0)$ is a {\em simple or reduced singularity} of $\widehat{\mathcal{F}_{\omega}}$ if the matrix associated with the linear part of the field
\begin{equation}\label{aaa}
\left( \begin{array}{ccc}
\frac{\partial \widehat{B}(0,0)}{\partial x}  &  \frac{\partial \widehat{B}(0,0)}{\partial y}\\
& & \\
 -\frac{\partial\widehat{A}(0,0)}{\partial x}  & -\frac{\partial \widehat{A}(0,0)}{\partial y}
\end{array} \right)
\end{equation}
has two eigenvalues $\lambda,\mu$, with $\frac{\lambda}{\mu} \not\in \mathbb{Q}^{+}$.

It could happen that

\begin{enumerate}
\item [$a)$] $\lambda \mu \neq 0$ and $\frac{\lambda}{\mu} \not\in \mathbb{Q}^{+}$ in which case we will say that the {\em singularity is not degenerate} or
\item [$b)$] $\lambda \mu = 0$ and $(\lambda,\mu)\neq (0,0)$ in which case we will say that the singularity is a {\em saddle-node}.
\end{enumerate}

In the $b)$ case, the {\em strong separatrix} of a foliation
with singularity $P$ is an analytic invariant curve whose tangent at the singular point $P$
is the eigenspace associated with the non-zero eigenvalue of the matrix given in (\ref{aaa}).
The zero eigenvalue is associated with a formal separatrix called {\em weak separatrix}.\\

From now on $\pi: M \rightarrow (\mathbb{C}^{2},0)$ represents {\em the process of singularity reduction or desingularization} of $\widehat{\mathcal{F}_{\omega}}$ \cite{Mattei-Moussu}, obtained by a finite sequence of point blows-up, where  $\mathcal{D}:=\pi^{-1}(0)=\displaystyle\bigcup_{j=1}^{n}D_{j}$ is the {\em exceptional divisor}, which is a finite union of projective lines with normal crossing (that is, they are locally described by one or two regular and transversal curves). In this process, any separatrix of $\widehat{\mathcal{F}_{\omega}}$ is smooth, disjoint and transverse to $D_{j} \subset \mathcal{D}$, and it does not pass through a corner (intersection of two components of the divisor $\mathcal{D}$). Let $\widehat{\mathcal{F}}_{\omega}$ be a non-dicritical formal foliation and consider the minimal reduction of singularities $\pi:M \rightarrow (\mathbb{C}^{2},0)$ of $\widehat{\mathcal{F}}_{\omega}$ (this is, a reduction with the minimal number of blows-up that reduces the foliation). The {\em strict transform} of the foliation $\widehat{\mathcal{F}}_{\omega}$ is given by $\widehat{\mathcal{F}}'_{\omega}=\pi^{*}\widehat{\mathcal{F}}_{\omega}$  and the {\em exceptional divisor} is $\mathcal{D}=\pi^{-1}(0)$.\\

A foliation $\widehat{\mathcal{F}_{\omega}}$ is a {\em generalized curve} if in its reduction of singularities there are no saddle-node points.\\

If in the desingularization of $\widehat{\mathcal{F}}_{\omega}$, the exceptional divisor $\mathcal{D}$ at point $P$ contains the weak invariant curve of the saddle-node, then the singularity is called {\em saddle-node tangent}. Otherwise we will say that $\widehat{\mathcal{F}}_{\omega}$ is a {\em saddle-node transverse} to $\mathcal{D}$ at point $P$.\\

\begin{defin}
The foliation $\widehat{\mathcal{F}}_{\omega}$ is of the {\em second type} with respect to the divisor $\mathcal{D}$ if no singular points of $\widehat{\mathcal{F}}'_{\omega}$ is of  tangent node type.
\end{defin}

Non-dicritical foliations of the second type were studied by Mattei and Salem \cite{Mattei-Salem}, also by Cano, Corral and Mol \cite{Cano2-Corral-Mol} and in the dicritical case by Genzmer and Mol \cite{Genzmer-Mol} and Fern\'andez P\'erez-Mol \cite{Fernandez-Mol}. Mattei and Salem gave the next characterization of foliations of the second type in terms of the multiplicity of their union of formal separatrices:\\

\begin{teorema}\cite[Th\'eor\`eme 3.1.9]{Mattei-Salem}\label{caracterizacion segundo tipo}
 Let $\widehat{\mathcal{F}_{\omega}}$ be a non-dicritical foliation and let $\widehat{\mathcal{S}}(\widehat{\mathcal{F}_{\omega}}):\widehat{f}(x,y)=0$ be a reduced equation of its union of separatrices. Consider the minimal reduction of singularities $\pi:(M,D)\rightarrow (\mathbb{C}^{2},0)$ of $\widehat{\mathcal{F}_{\omega}}$. Then
\begin{enumerate}
\item $\pi$ is a reduction of singularities of $\widehat{\mathcal{S}}(\widehat{\mathcal{F}_{\omega}})$. Furthermore, if $\widehat{\mathcal{F}_{\omega}}$ is of the second type then $\pi$ is the minimal reduction of singularities of $\widehat{\mathcal{S}}(\widehat{\mathcal{F}_{\omega}})$.
\item $\mult(\widehat{\omega})\geq \mult(\widehat{\mathcal{S}}(\widehat{\mathcal{F}_{\omega}}))$ and the equality holds if and only if $\widehat{\mathcal{F}_{\omega}}$ is of the second type.
\end{enumerate}
\end{teorema}

The reciprocal of the first statement of Theorem \ref{caracterizacion segundo tipo} is not true, that is, if the reduction of singularities of the foliation and that of its union of separatrices coincide does not guarantee that the foliation is of the second type, as shown in the following example.

\begin{ejemplo}\label{contra1}
The union of the separatrices of the foliation $\mathcal{F}_{\omega}=(xy+y^{2})\dd x-x^{2}\dd y$ is $\mathcal{S}(\mathcal{F}_{\omega})=xy$. The foliation $\mathcal{F}_{\omega}$ and its union of separatrices are desingularized after a blow-up but the foliation is not of the second type because the strong separatrix that passes through the saddle-node is not contained in the divisor.
\end{ejemplo}

\section{Foliations and Newton polygones} \label{general1}

Non-dicritical generalized curve foliations are those in which no saddle-node points appear in their desingularization \cite{Seidenberg} and they have a finite number of separatrices \cite{Camacho-Lins-Sad}. These foliations were studied by Camacho, Lins Neto and Sad who proved that

\begin{teorema} \label{misma resolucion111}\cite[Theorem 2]{Camacho-Lins-Sad} Let $\mathcal{F}_{\omega}$ be a non-dicritical generalized curve foliation and $\mathcal{S}(\mathcal{F}_{\omega})$ its union of separatrices. Then $\mathcal{F}_{\omega}$ and $\mathcal{S}(\mathcal{F}_{\omega})$ have the same reduction of singularities.
\end{teorema}

Rouill\'e obtained the following result on non-dicritical generalized curve foliations. In \cite{Rouille1} it is indicated that Mattei reported that this result was known by Dulac \cite{Dulac}.

\begin{propo}\label{poligono Newton iguales} \cite[Proposition 3.8]{Rouille1} Let $\mathcal{F}_{\omega}$ be a non-dicritical generalized curve foliation and $\mathcal{S}(\mathcal{F}_{\omega}):f(x,y)=0$ be a reduced equation of its union of separatrices. Then $\mathcal{N}(\omega)=\mathcal{N}(df)$.
\end{propo}

The reciprocal of  Theorem \ref{misma resolucion111} and Proposition \ref{poligono Newton iguales} are not true, as  the following example shows:

\begin{ejemplo}\label{segundo tipo}
Consider $b \not \in \mathbb{Q}$ and the foliation defined by $\omega=((b-1)xy-y^3)\dd x+(xy-bx^2+xy^2)\dd y$. A reduced equation of its union of separatrices is $f(x,y)=xy(x-y)=0$. The foliation $\mathcal{F}_{\omega}$ and the curve $f(x,y)=0$ are desingularised after a blow-up. The Newton polygons  $\mathcal{N}(\omega)$ and $\mathcal{N}(f)$ are equal but $\mathcal{F}_{\omega}$ is not a curve generalized type foliation since a saddle-node point appears in its reduction of singularities.
\end{ejemplo}

In \cite[pag 532]{Brunella1} was introduced the {\em G\'omez-Mont-Seade-Verjovsky index} denoted by $GSV(\mathcal{F}_{\omega},\mathcal{S}(\mathcal{F}_{\omega}))$, where $\mathcal{F}_{\omega}:\omega=0$ and
$\mathcal{S}(\mathcal{F}_{\omega}):f(x,y)=0$ is a reduced equation of union of convergent separatrices of $\mathcal{F}_{\omega}$. For $f \in \mathbb{C}\{x,y\}$, there are $g, h \in \mathbb{C}\{x,y\}$, with $h$ and $f$ coprime, and an analytic one-form $\eta$ such that $g\omega=h \dd f + f\eta$. In \cite{Brunella1}, Brunella defines
$$GSV(\mathcal{F}_{\omega},\mathcal{S}(\mathcal{F}_{\omega}))=\frac{1}{2\pi i}\int_{\partial \mathcal{S}(\mathcal{F}_{\omega})}\frac{g}{h}d\left(\frac{h}{g}\right),$$
when $\mathcal{S}(\mathcal{F}_{\omega}):f=0$ is irreducible and $\omega=A(x,y)dx+B(x,y)dy$. We get
\[
GSV(\mathcal{F}_{\omega},\mathcal{S}(\mathcal{F}_{\omega}))=\ord_{t}\left(\frac{B}{f_{y}}(\gamma(t))\right),
\]
where  $\gamma(t)$ is a parametrization of $\mathcal{S}(\mathcal{F}_{\omega})$. Now, we remember some results on the index $GSV(\mathcal{F}_{\omega},\mathcal{S}(\mathcal{F}_{\omega}))$.

\begin{teorema}\label{GSV3}\cite[Th\'eor\`{e}me 3.3]{Cavalier-Lehmann}\cite[Proposition 7]{Brunella1} Let $\mathcal{F}_{\omega}$ be a non-dicritical foliation and let $\mathcal{S}(\mathcal{F}_{\omega}):f(x,y)=0$ be a reduced equation of its union of separatrices. Then $\mathcal{F}_{\omega}$ is a generalized curve foliation if and only if $GSV(\mathcal{F}_{\omega},\mathcal{S}(\mathcal{F}_{\omega}))=0$.
\end{teorema}

The next result on generalized curve foliations was obtained by Rouill\'e \cite{Rouille1} and it will be very useful in this paper. Denote by $\mathbb{C}[[t]]$ the ring of formal power series in the variable $t$ and coefficients in $\mathbb{C}$, and $\mathbb{C}\{t\}$ the subring of $\mathbb{C}[[t]]$ of convergent power series.

\begin{lema}\label{orden igual} \cite[Lemme 3.7]{Rouille1} Let $\mathcal{F}_{\omega_{1}}$ and $\mathcal{F}_{\omega_{2}}$ two non-dicritical generalized curve foliations with the same union of separatrices. If $\gamma(t)=(x(t),y(t))\in (\mathbb{C}\{t\})^2$ with $\gamma(0)=0$, then
\[\ord_{t}\gamma^{\ast}\omega_{1}=\ord_{t}\gamma^{\ast}\omega_{2}.\]
\end{lema}

We deduce from Example \ref{segundo tipo} that there are foliations having the same polygon as their union of separatrices but they are not generalized curve foliations. Our objective in this paper is to characterize the foliations having the same Newton polygon that its union of separatrices. They will be the non-dicritical foliations of the second type.\\

\section{Characterization of a foliation of the second type in terms of the Newton polygon} \label{characterization}
In this section, a new characterization of the second-type non-dicritical foliations is given in terms of the Newton polygon of the foliation and that of its union of separatrices.

\begin{lema}\label{L1} Let $\widehat{\mathcal{F}_{\omega}}$ be a non-dicritical foliation and $\widehat{f}(x,y)=0$ be a reduced equation of its union of separatrices. If $\mathcal{N}(\widehat{\omega})=\mathcal{N}(\widehat{f})$ then $\widehat{\mathcal{F}_{\omega}}$ is a foliation of the second type.
\end{lema}
\begin{proof}
Consider the foliation $\widehat{\mathcal{F}_{\omega}}$ given by $\widehat{\omega}=\displaystyle\sum_{i,j}\widehat{A}_{ij}x^{i-1}y^{j}\dd x + \displaystyle\sum_{i,j}\widehat{B}_{ij}x^{i}y^{j-1}\dd y$. Since $\mult(\widehat{\omega})=\min\{\ord(\widehat{A}),\ord(\widehat{B})\}$ then 
\begin{equation}\label{mul segundo tipo}
\begin{array}{lll}
  \mult(\widehat{\omega}) & = & \min\{i+j-1:(i,j) \in \mathcal{N(\widehat{\omega})}\} \\
   & = & \min\{i+j-1:(i,j) \in \mathcal{N}(\dd \widehat{f})\} \\
   & = & \mult(\dd \widehat{f}).
\end{array}
\end{equation}
From (\ref{mul segundo tipo}) and the second statement of Theorem \ref{caracterizacion segundo tipo} we finish the proof.
\end{proof}

As a consequence of Lemma \ref{L1} and Theorem \ref{caracterizacion segundo tipo} we conclude that if $\mathcal{N}(\widehat{\omega})=\mathcal{N}(\widehat{f})$ then the foliation $\widehat{\mathcal{F}_{\omega}}$ and its union of separatrices $\widehat{\mathcal{S}}(\widehat{\mathcal{F}_{\omega}})$ have the same resolution.

In the following proposition we generalize Lemma \ref{orden igual} to second type foliations.

\begin{propo}\label{pieza} Let $\widehat{\mathcal{F}_{\omega}}$ be a non-dicritical second type foliation and $\widehat{\mathcal{S}}(\widehat{\mathcal{F}_{\omega}})$ its union of separatrices .
If $\gamma(t)=(x(t),y(t))\in (\mathbf C[[t]])^2$, with $\gamma(0)=0$, then
\[
\ord_{t}\gamma^{\ast}\widehat{\omega}=\ord_{t}\gamma^{\ast}\dd \widehat{f}.
\]
\end{propo}
\begin{proof} If $\gamma(t)=(x(t),y(t))$ is a parameterization of a separatrix of $\widehat{\omega}$ and $\dd \widehat{f}$ then $\widehat{\omega}(\gamma(t)).\gamma'(t)=0=\dd \widehat{f}(\gamma(t)).\gamma'(t)$ and we conclude the proposition in such a case. \\

Suppose now that $\gamma(t)$ is not a parameterization of any separatrix of $\widehat{\omega}$ and $\dd \widehat{f}$. We proceed by induction on the number of blows-up $n$ needed in the process of the desingularization of  the foliation $\widehat{\mathcal{F}_{\omega}}$. First we suppose that the number of blows-up is $n=0$. Then the foliations defined by the one-forms $\widehat{\omega}$ and $\dd \widehat{f}$ are reduced. If $\widehat{\mathcal{F}_{\omega}}$ is a generalized curve foliation then the proposition follows from Lemma \ref{orden igual}.\\

Suppose now that $\widehat{\mathcal{F}_{\omega}}$ is a reduced foliation with a saddle-node. We can consider the formal form of the saddle-node given by the equation:
\begin{equation*}
-y^{p+1}\dd x+(1+\lambda y^{p})x\dd y\;\; \hbox{\rm with }p\geq 1 \;\; \hbox{\rm and } \;\; \lambda \in \mathbb{C},
\end{equation*}
which under a change of coordinates can be expressed as (see \cite[Page 66]{Camacho-Sad2})
\begin{equation*}
\widehat{\omega}=x(1+\lambda y^{p})\dd y-y^{p+1}\dd x,
\end{equation*}
and the reduced equation of its union of formal separatrices is given by $\widehat{f}(x,y)=xy$. We can write $\gamma(t)=(x(t),y(t))=(t^{a}n_{1}(t),t^{b}n_{2}(t))$, where $a, b$ are positive integers and $n_{i}(t)$ are units of $\mathbb{C}[[t]]$  (that is $n_i(0)\neq 0$ for $i=1,2$). We have
\[
\begin{array}{lll}
   \gamma^{*}\widehat{\omega}
   &=& [bt^{a+b-1}n_{1}(t)n_{2}(t)+t^{a+b}n_{1}(t)n'_{2}(t)\\
   &+&
   t^{a+b+pb-1}n_{1}(t)(n_{2}(t))^{p+1}(\lambda b-a)\\
   &+&\lambda t^{a+b+pb}n_{1}(t)(n_{2}(t))^{p}n'_{2}(t)-t^{a+b+pb}n'_{1}(t)(n_{2}(t))^{p+1}] \dd t,
\end{array}
\]
so $\mult(\gamma^{*}\widehat{\omega})=a+b-1$. On the other hand $\dd \widehat{f}=y\dd x+x\dd y,$ hence
\[
\begin{array}{lll}
   \gamma^{*}\dd \widehat{f}
   & = & t^{b}n_{2}(t)(at^{a-1}n_{1}(t)+t^{a}n'_{1}(t))\dd t\\ &+& t^{a}n_{1}(t)(bt^{b-1}n_{2}(t)+t^{b}n'_{2}(t))\dd t\\
   & = & [t^{a+b-1}n_{1}(t)n_{2}(t)(a+b)+
   t^{a+b}(n'_{1}(t)n_{2}(t)+n_{1}(t)n'_{2})]\dd t,
\end{array}
\]
so $\mult(\gamma^{*} \dd \widehat{f})=a+b-1$. Therefore, if $\widehat{\mathcal{F}_{\omega}}$ is a reduced foliation with a saddle-node, then $\ord_{t}(\gamma^{*}\widehat{\omega})=\ord_{t}(\gamma^{*}\dd \widehat{f})$.\\

Now we suppose that the foliations defined by the one-forms $\widehat{\omega}$ and $\dd \widehat{f}$ are not reduced and $n>0$. Let $E$ be the blow-up at the origin $(x,y)=(0,0)$ given by $E:(x,t)=(x,xt)$, so $E^{*}\widehat{\omega}=x^{m_{1}}\widetilde{\widehat{\omega}}$, where $m_{1}$ is the multiplicity of $\widehat{\omega}$ and $\widetilde{\widehat{\omega}}$ is the strict transform of $\widehat{\omega}$. Denote by $\widetilde{\gamma}$ the strict transformation  of the curve $\gamma$ by $E$. By  induction hypothesis, we get,
\[
\ord_{t}\widetilde{\gamma}^{*}\widetilde{\widehat{\omega}}=\ord_{t}\widetilde{\gamma}^{*}\widetilde{\dd \widehat{f}}.
\]

On the other hand we have
$\begin{array}{lll}
    \gamma^{*}\widehat{\omega} = x(t)^{m_{1}}\widetilde{\gamma}^{*}\widetilde{\widehat{\omega}},
\end{array}$
hence \begin{equation} \label{orden w1 y w2}
\begin{array}{lll}
    \ord_{t}\gamma^{*}\widehat{\omega} = \mult(x(t))\mult(\widehat{\omega})+\ord_{t}\widetilde{\gamma}^{*}\widetilde{\widehat{\omega}}.
  \end{array}
\end{equation}
Since the foliation $\widehat{\mathcal{F}_{\omega}}$ is of the second type, by  Theorem \ref{caracterizacion segundo tipo}, using the induction hypothesis and replacing in the equation (\ref{orden w1 y w2}) we get $\ord_{t}\gamma^{\ast}\widehat{\omega}=\ord_{t}\gamma^{\ast}\dd \widehat{f}$ and we finish the proof of the proposition.
\end{proof}

Proposition \ref{pieza} was also given in \cite[Corollary 1]{Cano2-Corral-Mol}, but with other proof.

\begin{coro}\label{poligonoigual222}Let $\widehat{\mathcal{F}_{\omega}}$ be a non-dicritical second type foliation and let $\widehat{\mathcal{S}_{f}}$ be a reduced equation of its union of formal separatrices. Then $\mathcal{N}(\widehat{\omega})=\mathcal{N}(\dd \widehat{f})$.
\end{coro}
\begin{proof} Since $\widehat{\mathcal{F}_{\omega}}$ is a second type foliation, using
Theorem \ref{caracterizacion segundo tipo}, we have that $\widehat{\mathcal{F}_{\omega}}$ has the same reduction of singularities as its union of formal separatrices and $\mult(\widehat{\omega})=\mult(\dd \widehat{f})$. Reasoning analogously as in the proof given by Rouill\'e \cite{Rouille1} in order to prove the Proposition \ref{poligono Newton iguales}, and by Proposition \ref{pieza}, we finish the proof.
\end{proof}

As a consequence of the Corollary \ref{poligonoigual222} we have:

\begin{coro} Let $\widehat{\mathcal{F}_{\omega_{1}}}$ and $\widehat{\mathcal{F}_{\omega_{2}}}$ be two non-dicritical second type foliations. If  $\widehat{\mathcal{F}_{\omega_{1}}}$ and $\widehat{\mathcal{F}_{\omega_{2}}}$ have the same union of formal separatrices, then $\mathcal{N}(\widehat{\omega_{1}})=\mathcal{N}(\widehat{\omega_{2}})$.
\end{coro}

\begin{ejemplo} The foliation $\mathcal{F}_{\omega}$ given by $\omega=(ny+x^{n})\dd x-x\dd y$, $n \geq 1$ is not a foliation of the second type. The union of separatrices of $\mathcal{F}_{\omega}$ is $\mathcal{S}(\mathcal{F}_{\omega}):x=0$. We observe that $\sop(\omega)=\{(1,1),(n+1,0)\}$ and $\sop(f)=\{(1,0)\}$, hence the Newton polygons  of  $\mathcal{F}_{\omega}$ and
$\mathcal{S}(\mathcal{F}_{\omega})$ are different.

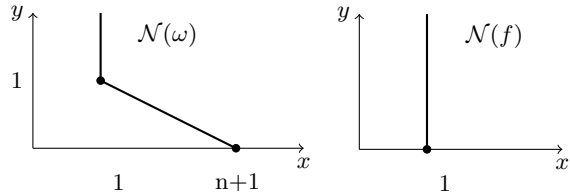
\begin{center}
\begin{tikzpicture}[x=0.9cm,y=0.9cm]
\tikzstyle{every node}=[font=\small]
 (1,2) -- (1,1) -- (n+1,0) -- cycle;
\draw[->] (0,0) -- (4,0) node[right,below] {$x$};
\draw[->] (0,0) -- (0,2) node[above,left] {$y$};
\draw[thick] (0,1) node[left] {1};
\draw[thick] (1.5,-0.5) node[left] {1};
\node[draw,circle,inner sep=1pt,fill] at (1,1) {};
\draw[thick] (3.5,-0.5) node[left] {n+1};
\node[draw,circle,inner sep=1pt,fill] at (3,0) {};
\draw[thick] (2,2) node[below] {$\mathcal{N}(\omega)$};
\draw[thick] (1,2) -- (1,1);
\draw[thick] (1,1) -- (3,0);
\end{tikzpicture}
\begin{tikzpicture}[x=0.9cm,y=0.9cm]
\tikzstyle{every node}=[font=\small]
 (1,2) -- (1,1) --  cycle;
\draw[->] (0,0) -- (3,0) node[right,below] {$x$};
\draw[->] (0,0) -- (0,2) node[above,left] {$y$};
\draw[thick] (1.5,-0.5) node[left] {1};
\node[draw,circle,inner sep=1pt,fill] at (1,0) {};
\draw[thick] (2,2) node[below] {$\mathcal{N}(f)$};
\draw[thick] (1,2) -- (1,0);
\end{tikzpicture}
\end{center}

\end{ejemplo}

\begin{ejemplo}\label{segundo tipo1} Let us go back to Example \ref{segundo tipo}. The second type foliation given by
$\omega=((b-1)xy-y^3)dx+(xy-bx^2+xy^2)dy$ with $-b, 1-b \not\in \mathbb{Q}^{+}$ has as union of separatrices to $\mathcal{S}(\mathcal{F})=xy(x-y)$. We observe that polygons $\mathcal{N}(\omega)$ and $\mathcal{N}(f)$ are equal.

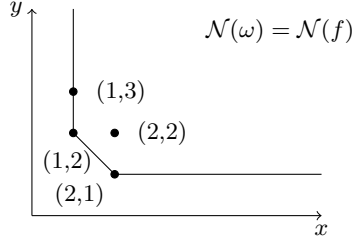
\begin{center}
\begin{tikzpicture}[x=0.55cm,y=0.55cm]
\tikzstyle{every node}=[font=\small]
\draw[->] (0,0) -- (7,0) node[right,below] {$x$};
\draw[->] (0,0) -- (0,5) node[above,left] {$y$};
(1,5) -- (1,2) -- (2,1) -- (2,7) -- cycle;
\draw (1,5) -- (1,2);
\draw (1,2) -- (2,1);
\draw (2,1) -- (7,1);
\node[draw,circle,inner sep=1pt,fill] at (1,2) {};
\draw[thick] (1.7,1.3) node[left] {(1,2)};
\node[draw,circle,inner sep=1pt,fill] at (2,1) {};
\draw[thick] (2,0.5) node[left] {(2,1)};
\node[draw,circle,inner sep=1pt,fill] at (1,3) {};
\draw[thick] (3,3) node[left] {(1,3)};
\node[draw,circle,inner sep=1pt,fill] at (2,2) {};
\draw[thick] (4,2) node[left] {(2,2)};
\draw[thick] (6,5) node[below] {$\mathcal{N}(\omega)=\mathcal{N}(f)$};
\end{tikzpicture}
\end{center}

\end{ejemplo}

{\bf Proof of Theorem \ref{resultado1}.}
It is an  immediate consequence of Corollary \ref{poligonoigual222} and Lemma \ref{L1}.
\fdp
\\

Theorem \ref{resultado1}  gives a new characterization of the non-dicritical second type foliations using its Newton polygon.

\section{Cuspidal Foliations}\label{cuspidal section}
{\em Cuspidal foliations} are inspired by {\em nilpotent foliations.} A foliation $\mathcal{F}_{\omega}$ in $(\mathbb{C}^{2},0)$ is called a nilpotent singularity if it is generated by a vector field $X$ with a non-zero nilpotent linear part (that is, the matrix associated with the linear part of the field is nilpotent). The nilpotent singularities were generalized to cuspidal singularities by Loray \cite{Loray}, as we shall see below.\\

In this section we characterize when foliations with cuspidal singularities are of the second type in terms of weighted order. Furthermore, by means of the weighted order, we give necessary and sufficient conditions for these foliations to be of generalized curve type.

Given $p, q \in \mathbb{N}^{*}$, we define the {\em weighted degree} of a monomial $x^iy^j$ as

\[
\grad_{(p,q)}(x^{i}y^{j})=\frac{ip+jq}{\gcd(p,q)},
\]

and the {\em weighted order} of  a power series $f(x,y)=\displaystyle\sum_{i,j}a_{i,j}x^{i}y^{j}\in \mathbf C\{x,y\}$ as

\[
\ord_{(p,q)}(f(x,y))=\min\left\{\grad_{(p,q)}(x^{i}y^{j}):a_{i,j}\neq 0\right\}.
\]

\medskip

Remember that according to Loray \cite{Loray}, a foliation with a cuspidal singularity is given as in (\ref{cusp}), that is by
\[
\mathcal{F}_{\omega_{p,q,\Delta}}:\:\omega_{p,q,\Delta}=\dd(y^{p}-x^{q})+\Delta(x,y)(px\dd y-qy\dd x),
\]
where $p,q$ are positive natural numbers and $\Delta(x,y) \in \mathbb{C}\{x,y\}$.

\medskip

On the other hand, remember that  $\PH_{(p,q)}:=\PH(\dd(y^{p}-x^{q}))=pq-p-q+1$.

\medskip

Cuspidal foliations are nilpotent foliations when $p = 2$.

\medskip

For Loray, the foliations $\mathcal{F}_{\omega_{p,q,\Delta}}$ and $\dd(y^p-x^q)$ have the same resolution of singularities if and only if $\ord_{(p,q)}(\Delta)> \frac{pq-p-q}{\gcd(p,q)}=\frac{\PH_{(p,q)}-1}{\gcd(p,q)}$. Fern\'andez, Mozo and Neciosup \cite{Fernandez-Mozo-Neciosup}, find an imprecision in the characterization originally proposed by Loray. These authors mention that the condition is sufficient but not necessary, as can be seen from the following example.

\begin{eje}\label{contraejemplo Loray} The foliation $\omega=\dd(y^{6}-x^{3})+axy(6x\dd y-3y \dd x)$ with $a \not\in \{-(6r+1)\zeta/r \in \mathbb{Q}_{>0}\;\; y \;\;\zeta^{3}=1\}$ has the same resolution as the foliation $\dd (y^{6}-x^{3})=0$, but the function $\Delta(x,y)=axy$ satisfies $\ord_{(6,3)}\Delta=3$, so the inequality $\ord_{(6,3)}\Delta > \frac{\PH_{(p,q)}-1}{\gcd(p,q)}$ does not hold.
\end{eje}

For $d=\gcd(p,q)$, we have
$$y^{p}-x^{q}=\displaystyle\prod_{i=1}^{d}(y^{\frac{p}{d}}-\zeta^{i}x^{\frac{q}{d}}),$$
and $\gamma_{i}(t)=(t^{\frac{p}{d}},A_{i}t^{\frac{q}{d}})$ with $A_{i}^{\frac{p}{d}}=\zeta^{i}$ is a parameterization of $\mathcal{S}_{i}:f_{i}(x,y)=(y^{\frac{p}{d}}-\zeta^{i}x^{\frac{q}{d}})$, with $\zeta \in \mathbb{C}, \; \zeta^{d}=1$. We get
$$(\Delta,y^{p}-x^{q})_{0}=\displaystyle\sum_{i}(\Delta,f_{i})_{0}=d\cdot \ord_{(p,q)}(\Delta),$$
where $(f,g)_{0}=\dim_{\mathbb{C}}\mathbb{C}\{x,y\}/(f,g)$ is the intersection number of $f$ and $g$.

\begin{lema}\label{eje separatriz4} If the cuspidal foliation $\mathcal{F}_{\omega_{p,q,\Delta}}:\omega_{p,q,\Delta}=0$ is a non-dicritical foliation, then $\mathcal{S}(\mathcal{F}_{\omega_{p,q,\Delta}})=y^{p}-x^{q}=0$  is its union of separatrices.
\end{lema}
\begin{proof}
The curve $\mathcal{S}_{f}:y^{p}-x^{q}=0$ is an invariant curve of the foliation $\mathcal{F}_{\omega_{p,q,\Delta}}$. Put $\alpha=\ord(\Delta)$. Then

\begin{equation}\label{multiplicidad de folia-Loray}
\mult(\omega_{p,q,\Delta})=\min\{q-1, p-1,\alpha + 1\}.
\end{equation}

\medskip

Suppose that $p<q$. The multiplicity of the curve  $\mathcal{S}_{f}$ is $p$. If we assume that the curve $\mathcal{S}_{f}$ is not the only separatrix of the foliation $\mathcal{F}_{\omega_{p,q,\Delta}}$, then $\mult(\mathcal{S}(\mathcal{F}_{\omega_{p,q,\Delta}}))>p$. Using \eqref{multiplicidad de folia-Loray}, we have
$\mult(\omega_{p,q,\Delta})=\min\{p-1,\alpha + 1\}.$
We will study both possibilities:
\begin{enumerate}
\item [(i)] If $\mult(\omega_{p,q,\Delta})=p-1$ then $p-1=\mult(\omega_{p,q,\Delta})\geq \mult(\mathcal{S}(\mathcal{F}_{\omega_{p,q,\Delta}}))-1>p-1$, which is a contradiction.
\item [(ii)] If $\mult(\omega_{p,q,\Delta})=\alpha+1$ then $\alpha+1=\mult(\omega_{p,q,\Delta})\geq \mult(\mathcal{S}(\mathcal{F}_{\omega_{p,q,\Delta}}))-1>p-1,$ which is a contradiction since $\alpha+1\leq p-1$.
\end{enumerate}
Therefore the union of separatrices of the foliation $\mathcal{F}_{\omega_{p,q,\Delta}}$ is $\mathcal{S}(\mathcal{F}_{\omega_{p,q,\Delta}})=y^{p}-x^{q}$. The same reasoning happens when $p\geq q$ and we conclude that $\mathcal{S}(\mathcal{F}_{\omega_{p,q,\Delta}})=y^{p}-x^{q}$.
\end{proof}

\begin{propo}\label{orden pesado-segundo tipo1}  Suppose that  the cuspidal foliation $\mathcal{F}_{\omega_{p,q,\Delta}}:\omega_{p,q,\Delta}=0$ is non-dicritical. If $(\Delta,y^{p}-x^{q})_{0} \geq \PH_{(p,q)}-1$, with $\dd=\gcd(p,q)$, then the foliation $\mathcal{F}_{\omega_{p,q,\Delta}}$ is of the second type.
\end{propo}
\begin{proof}
Suppose without lost of generality  that $p<q$ and $\ord \Delta=i_{0}+j_{0}.$
Since $(\Delta,y^{p}-x^{q})_{0} \geq \PH_{(p,q)}-1$ we have
$i_{0}p+j_{0}q \geq \PH_{(p,q)}-1$. After $p<q$ we get
\[
i_{0}q+j_{0}q>i_{0}p+j_{0}q \geq \PH_{(p,q)}-1,
\]
so $i_{0}+j_{0}>p-1-\frac{p}{q}>p-2$ and $\alpha=\ord \Delta \geq p-1$.
Since $\mult(\dd f)=p-1$ for $\mathcal{S}(\mathcal{F}_{\omega_{p,q,\Delta}}):f(x,y)=y^{p}-x^{q}=0$, using (\ref{multiplicidad de folia-Loray}) we have $\mult (\omega_{p,q,\Delta})=p-1$. Hence $\mult (\omega_{p,q,\Delta})=\mult(\dd f)$ and we conclude that the foliation $\mathcal{F}_{\omega_{p,q,\Delta}}$  is of the second type.
\end{proof}

\begin{propo}\label{orden pesado-segundo tipo2} Suppose that  the cuspidal foliation $\mathcal{F}_{\omega_{p,q,\Delta}}:\omega_{p,q,\Delta}=0$ is non-dicritical. If $\mathcal{F}_{\omega_{p,q,\Delta}}$
is the second type, then $(\Delta,y^{p}-x^{q})_{0} \geq \PH_{(p,q)}-1.$
\end{propo}
\begin{proof} From Lemma \ref{eje separatriz4} we get $\mathcal{S}(\mathcal{F}_{\omega_{p,q,\Delta}})=y^{p}-x^{q}$. Put $d:=\gcd(p,q)$. The line containing the only compact side of Newton polygone $\mathcal{N}(\dd f)$ is $\mathcal{L}:\frac{p}{d}i+\frac{q}{d}j=\frac{pq}{d}$. Since $\mathcal{F}_{\omega_{p,q,\Delta}}$ is of the second type, using Theorem \ref{resultado1} we have  $\mathcal{N}(\omega_{p,q,\Delta})=\mathcal{N}(f)$. Therefore, the line $\mathcal{L}$ also contains the only compact side of the Newton polygon of $\mathcal{N}(\omega_{p,q,\Delta})$, that is any $(a,b) \in \sop(\omega_{p,q,\Delta})$ verifies $a\frac{p}{d}+b\frac{q}{d} \geq \frac{pq}{d}$. Suppose that
$\Delta(x,y)=\displaystyle\sum_{i,j}a_{ij}x^{i}y^{j} \in \mathbb{C}\{x,y\}$, then

\[
\omega_{p,q,\Delta}=\left(-qx^{q-1}-q\displaystyle\sum_{i,j}a_{ij}x^{i}y^{j+1}\right)\dd x+
\left(py^{p-1}+p\displaystyle\sum_{i,j}a_{ij}x^{i+1}y^{j}\right)\dd y,
\]
and $\sop(\omega_{p,q,\Delta})=\{(q,0),(i+1,j+1)\}\cup \{(0,p)(i+1,j+1)\},$ for $(i,j) \in \sop(\Delta)$. If $(i+1,j+1) \in \sop(\omega_{p,q,\Delta})$ then $(i+1)\frac{p}{d}+(j+1)\frac{q}{d} \geq \frac{pq}{d}$, so we conclude  that $(\Delta,y^{p}-x^{q})_{0} =ip+jq \geq \PH_{(p,q)}-1$.
\end{proof}

\begin{propo} \label{segundo tipo misma resolucion}  Suppose that  the cuspidal foliation $\mathcal{F}_{\omega_{p,q,\Delta}}:\omega_{p,q,\Delta}=0$ is non-dicritical. The foliation $\mathcal{F}_{\omega_{p,q,\Delta}}$ has the same reduction of singularities that $\dd (y^p-x^q)$, if and only if, $(\Delta,y^{p}-x^{q})_{0} \geq \PH_{(p,q)}-1.$
\end{propo}

\begin{proof} Suppose that $(\Delta,y^{p}-x^{q})_{0} \geq \PH_{(p,q)}-1$. By Proposition \ref{orden pesado-segundo tipo1}  the foliation $\mathcal{F}_{\omega_{p,q,\Delta}}$ is of the second type and by Theorem \ref{caracterizacion segundo tipo} we conclude that $\mathcal{F}_{\omega_{p,q,\Delta}}$ and $\dd (y^p-x^q)$ have the same reduction of singularities.

Suppose now that $\mathcal{F}_{\omega_{p,q,\Delta}}$ and $\dd (y^p-x^q)$ have the same reduction of singularities. The curve $y^{p}-x^{q}=0$ with $p > q$ and $d=\gcd(p,q)$ is desingularized by
\begin{equation}\label{resolucion torica17}
E:(x,y)=(u^{n}v^{\frac{p}{d}},u^{m}v^{\frac{q}{d}}),
\end{equation}
such that $mp-nq=d$  and $m, n \in \mathbb{N}^{*}$. The transformation of
\[
\omega_{p,q,\Delta}=(-qx^{q-1}-qy\Delta(x,y))\dd x+(py^{p-1}+px\Delta(x,y))\dd y,
\]
by $E$ defined as \eqref{resolucion torica17} is

\begin{eqnarray}\label{foliacion-resolucion torica211}
E^{*}\omega_{p,q,\Delta} & = &
 \left [ u^{nq-1}v^{\frac{pq}{d}}\left(-nq+mpu^{mp-nq}\right)+dv^{\frac{pq}{d}}u^{nq-1}(u^{m+n-nq}v^{\frac{p+q-pq}{d}}E^{*}(\Delta(x,y)))\right ]\dd u \nonumber \\
   &+&\left [ \frac{pq}{d}u^{nq}v^{\frac{pq}{d}-1}(-1+u^{mp-nq})\right ] \dd v
   \nonumber\\
   &=&\left( u^{nq-1}v^{\frac{pq}{d}-1}\right) v (-qn+mp u^{d}+\widetilde{\Delta}
   (u,v) )\dd u +\frac{pq}{d} u (u^{d}-1) \dd v,
\end{eqnarray}
where
\[
\begin{array}{lll}
\widetilde{\Delta}(u,v))&=& d E^{\ast}(\Delta(x,y))u^{m+n-nq}v^{\frac{p+q-pq}{d}}\\
&=&\displaystyle\sum_{i,j} d a_{ij} u^{ni+mj+m+n-nq}v^{\frac{pi+qj+p+q-pq}{d}}.
\end{array}
\]

Hence
\begin{equation}
\label{eeee}
\begin{array}{lll}
\dfrac{E^{*}\omega_{p,q,\Delta} }{u^{nq-1}v^{pq-1}} &=& v(-qn+mp u^{d}+\widetilde{\Delta}(u,v))\dd u+
   \frac{pq}{d} u (u^{d}-1)]\dd v,
\end{array}
\end{equation}

which singularities  are $(0,0)$ and $(\zeta^{j},0)$, where  $\zeta$  is a $d$th-primitive root of the unity. The dual vector field associated with the foliation defined by \eqref{eeee} is
\[
X=\frac{pq}{d}(u^{d}-1)u \frac{\partial}{\partial u} - v(-nq+mp u^{d}+\widetilde{\Delta}(u,v))\frac{\partial}{\partial v},
\]
and the matrix associated with this field is
\[
DX=\left(
    \begin{array}{cc}
      -\frac{pq}{d}+\frac{(d+1)pq}{d}u^{d} & 0 \\
       \ast & nq-mpu^{d}-\widetilde{\Delta}(u,v)-v\frac{\partial \widetilde{\Delta}(u,v)}{\partial v} \\
    \end{array}
  \right).
  \]

\begin{enumerate}
\item In $(0,0)$ we have
$DX=\left(
    \begin{array}{cc}
      -\frac{pq}{d} & 0 \\
       \ast & nq \\
    \end{array}
  \right).$ Therefore, the singularity $(0,0)$ is not degenerate.
\item If $u^{d}=1$ and $v=0$ then we get
$DX=\left(
    \begin{array}{ccc}
      -\frac{pq}{d} & 0 \\
       \ast & -d-\widetilde{\Delta}(u,v)\\
    \end{array}
  \right).$ Since the foliation is reduced, it could happen that
\begin{itemize} \label{orden pesadov}
\item $-d-\widetilde{\Delta}(\zeta^{j},0)=0$, from where $\widetilde{\Delta}(\zeta^{j},0)=-d$, in which case the singularity is of saddle-node type.
\item $-d-\widetilde{\Delta}(\zeta^{j},0)=-a,$ so that $\lambda=\frac{pq}{-a} \not \in \mathbb{Q}^{+}$, in this case, the singularity is of a no degenerate type.
\end{itemize}
We conclude that $\ord_{v}\widetilde{\Delta}\geq 0$, so $pi+qj+p+q \geq 0$ for some $(i,j)$. Therefore $(\Delta,y^{p}-x^{q})_{0} \geq \PH_{(p,q)}-1$ for some $(i,j) \in \sop(\Delta)$.
\end{enumerate}
\end{proof}

{\bf Proof of Theorem \ref{resultado2}.} The equivalence of statements $(a)$ and $(b)$ is a direct consequence of  Propositions \ref{orden pesado-segundo tipo1} and \ref{orden pesado-segundo tipo2}. The equivalence of statements $(b)$ and $(c)$ is   Proposition \ref{segundo tipo misma resolucion}.
\fdp

\begin{coro}\label{corocurvagenera}
Suppose that  the cuspidal foliation $\mathcal{F}_{\omega_{p,q,\Delta}}:\omega_{p,q,\Delta}=0$ is non-dicritical.
 If the foliation $\mathcal{F}_{\omega_{p,q,\Delta}}:\omega_{p,q,\Delta}=0$ is of the generalized curve type then $(\Delta,y^{p}-x^{q})_{0} \geq \PH_{(p,q)}-1$.
\end{coro}

\medskip

 The fact that the foliation $\mathcal{F}_{\omega_{p,q,\Delta}}$ is of generalized curve type does not imply that $(\Delta,y^{p}-x^{q})_{0} > \PH_{(p,q)}-1$, as the next example shows:

\medskip

\begin{eje}\label{sisepuede}
The foliation
$$
\omega=d(y^6-x^3)+axy(6xdy-3ydx),
$$
with $a \in \{-(6r+1)\zeta/r \in \mathbb{Q}_{> 0}\;\; y \;\;\zeta^{3}=1\} \subseteq \mathbb{C}^{\ast},$
and $a^{3}\neq -1$ is of the generalized curve type, but $(\Delta,y^{p}-x^{q})_{0}=3=\PH_{(p,q)}-1$, where $p=6, q=3$ and $d=3$.
\end{eje}

\medskip

Nevertheless

\begin{propo}\label{curva generalizada 2} Suppose that  the cuspidal foliation $\mathcal{F}_{\omega_{p,q,\Delta}}:\omega_{p,q,\Delta}=0$ is non-dicritical and $p$ and $q$ are coprimes. The foliation $\mathcal{F}_{\omega_{p,q,\Delta}}$ is of generalized curve type, if and only if $(\Delta,y^{p}-x^{q})_{0} >\PH_{(p,q)}-1$.
\end{propo}

\begin{proof}
Let us consider $\omega_{p,q,\Delta}=(-qx^{q-1}-qy \Delta)\dd x+(py^{p-1}+px \Delta)\dd y$, $f(x,y)=y^{p}-x^{q}$ and $\gamma(t)=(t^{p},t^{q})$ a parameterization of $f(x,y)=0$. Thus
\begin{equation}\label{GSV-curva}
\begin{array}{lll}
GSV(\mathcal{F}_{\omega_{p,q,\Delta}},\mathcal{S}(\mathcal{F}_{\omega_{p,q,\Delta}}))&=& \ord_{t}\left(\frac{py^{p-1}+px \Delta}{py^{p-1}}(t^{p},t^{q})\right)
=\ord_{t}\left(1+\frac{t^{p}\Delta(t^{p},t^{q})}{t^{q(p-1)}}\right),
\end{array}
\end{equation}
where  $\Delta(t^{p},t^{q})=\displaystyle\sum_{ij}a_{ij}t^{pi+qj}$. Note that
\[
GSV(\mathcal{F}_{\omega_{p,q,\Delta}},\mathcal{S}(\mathcal{F}_{\omega_{p,q,\Delta}}))=0,\; \mbox{ if and only if} \; \ord_{t}\left(1+\frac{t^{p}\Delta(t^{p},t^{q})}{t^{q(p-1)}}\right)=0,
\]
what is equivalent to  $(\Delta,y^{p}-x^{q})_{0} > \PH_{(p,q)}-1$. From Theorem \ref{GSV3} we conclude that $\mathcal{F}_{\omega_{p,q,\Delta}}$ is of the generalized curve type, if and only if $(\Delta,y^{p}-x^{q})_{0} > \PH_{(p,q)}-1$.
\end{proof}

\medskip

Suppose now that $p$ and $q$ are not coprime. We will analyze if the strict inequality $(\Delta,y^{p}-x^{q})_{0} > \PH_{(p,q)}-1$ is a sufficient condition for $\mathcal{F}_{\omega_{p,q,\Delta}}$ to be a foliation is of generalized curve type. We begin studying what happens when $d=\gcd(p,q)=2$.

\medskip

Let us consider $\mathcal{S}:f=f_{1}f_{2}$ and
$
g\omega = h\dd(f_{1}f_{2})+f_{1}f_{2}\eta.
$
For $\mathcal{S}_{i}:f_{i}(x,y)=0$, we have
$
\begin{array}{lll}
GSV(\mathcal{F},\mathcal{S}_{1})= \frac{1}{2\pi i} \displaystyle\int_{\partial \mathcal{S}_{1}}\frac{\dd(\frac{h}{g})}{\frac{h}{g}}+(f_{2},f_{1})_{0}.
\end{array}
$
Analogously, $GSV(\mathcal{F},\mathcal{S}_{2})=\frac{1}{2\pi i} \displaystyle\int_{\partial \mathcal{S}_{2}}\frac{\dd(\frac{h}{g})}{\frac{h}{g}}+(f_{1},f_{2})_{0}.$ We have
$$\frac{1}{2\pi i} \displaystyle\int_{\partial \mathcal{S}_{1} \cup \partial \mathcal{S}_{2}}\frac{\dd(\frac{h}{g})}{\frac{h}{g}}=GSV(\mathcal{F},\mathcal{S}_{1})+GSV(\mathcal{F},\mathcal{S}_{2})-2(f_{1},f_{2})_{0}.$$
Therefore (see \cite[page 532]{Brunella1}),
\begin{equation}\label{GSV-varias ramas}
GSV(\mathcal{F},\mathcal{S})=GSV(\mathcal{F},\mathcal{S}_{1})+GSV(\mathcal{F},\mathcal{S}_{2})-2(f_{1},f_{2})_{0}.
\end{equation}

For $d=2=\gcd(p,q)$, we have
$$y^{p}-x^{q}=\displaystyle\prod_{i=1}^{2}(y^{\frac{p}{2}}-\zeta^{i}x^{\frac{q}{2}}), \;\mbox{ with }\; \zeta^{2}=1.$$
Let $\mathcal{S}_{i}:f_{i}(x,y)=(y^{\frac{p}{2}}-\zeta^{i}x^{\frac{q}{2}})$ and $\gamma_{i}(t)=(t^{\frac{p}{2}},A_{i}t^{\frac{q}{2}})$ with $A_{i}^{\frac{p}{2}}=\zeta^{i}$ a parameterization of $\mathcal{S}_{i}$. Then

\begin{equation}
\begin{array}{lll}\label{indice sep}
(f_{1},f_{2})_{0}= \ord_{t}(f_{1}(\gamma_{2}(t)))
= \ord_{t}(t^{\frac{pq}{4}}(1-\zeta))
= \frac{pq}{4}.
\end{array}
\end{equation}

Remember that $\omega_{p,q,\Delta}=(-qx^{q-1}-qy \Delta)\dd x+(py^{p-1}+px \Delta)\dd y$, thus
\begin{equation}
\begin{array}{lll}\label{GSV coprimo}
GSV(\mathcal{F},\mathcal{S}_{1}) = \ord_{t}\left(t^{\frac{pq}{4}}(\zeta -1)+\frac{(\zeta -1)}{A_{1}^{p-1}}t^{\frac{p}{2}+\frac{q}{2}-\frac{pq}{4}}\Delta(t^{\frac{p}{2}},A_{1}t^{\frac{q}{2}})\right).
\end{array}
\end{equation}

If we consider $(\Delta,y^{p}-x^{q})_{0} > \PH_{(p,q)}-1$, from (\ref{GSV coprimo}) we have that $GSV(\mathcal{F},\mathcal{S}_{1})=\frac{pq}{4}$. Similarly, it turns out that $GSV(\mathcal{F},\mathcal{S}_{2})=\frac{pq}{4}$.

After (\ref{indice sep}) and (\ref{GSV-varias ramas}) we have $GSV(\mathcal{F},\mathcal{S})=0$, which is equivalent to $\omega_{p,q,\triangle}$, so the foliaction $\mathcal{F}$ generalized curve type when $d=2$.

In general \cite{Brunella1}, when $\mathcal{S}:f=f_{1}\cdots f_{d}$, we have to
\begin{equation}\label{GSV d ramas}
GSV(\mathcal{F},\mathcal{S})=\displaystyle\sum_{i=1}^{d}GSV(\mathcal{F},\mathcal{S}_{i})-2 \displaystyle\sum^{N}_{\substack {i\neq j\\
i=1}} (f_{i},f_{j})_{0},
\end{equation}
where $N=
\left(
\begin{array}{cc}
d\\
2\\
\end{array}
\right)$,\;
$GSV(\mathcal{F},\mathcal{S}_{i})=\dfrac{(d-1)pq}{d^{2}}$, and $(f_{i},f_{j})_{0}=\dfrac{pq}{d^{2}}$. Therefore, from (\ref{GSV d ramas}) we get
\begin{equation}\label{GSV dicritica}
GSV(\mathcal{F},\mathcal{S})=0.
\end{equation}
Hence the following proposition holds.

\begin{propo}\label{GSVd 2} Let $\mathcal{F}_{p,q,\Delta}$ be a non dicritical foliation and suppose that $(\Delta,y^{p}-x^{q})_{0} > \PH_{(p,q)}-1$. Then $\mathcal{F}_{p,q,\Delta}$ is of the generalized curve type.
\end{propo}

\medskip

{\bf Proof of Theorem \ref{resultado3}.}
It is an  immediate consequence of Propositions \ref{curva generalizada 2} and \ref{GSVd 2}.
\fdp
\\

\end{document}